\documentclass[12pt]{amsart}
\usepackage{enumerate}
\usepackage{caption}
\usepackage{tikz}
\usepackage{hyperref}
\usepackage{float}
\usepackage{longtable}
\usepackage{blindtext}
\usepackage{adjustbox}
\allowdisplaybreaks
\usepackage{longtable}
\usepackage{amsfonts,amssymb,amsmath,textcomp}
\usepackage{enumitem}
\usepackage{orcidlink}
\usepackage{lipsum}
\usepackage{supertabular}
\usepackage[margin=1in]{geometry}
 2

\theoremstyle{plain}
\newtheorem{thm}{Theorem}[section]

\newtheorem{prop}[thm]{Proposition}

\theoremstyle{definition}
\newtheorem{defn}[thm]{Definition}

\begin{document}
\baselineskip 6.1mm
\title{A primality test for $Kp^\ell - 1$ numbers }

\author[Anuj Jakhar]{Anuj Jakhar\,\orcidlink{0000-0002-8733-0007}}
\author[Mahesh Kumar Ram]{Mahesh Kumar Ram\,\orcidlink{0000-0002-5692-5761}}

\address[Anuj Jakhar, Mahesh Kumar Ram]{Department of Mathematics, IIT Madras, Chennai, India - 600036}

\email[Anuj Jakhar]{anujjakhar@iitm.ac.in; anujiisermohali@gmail.com}
\email[Mahesh Kumar Ram]{maheshkumarram621@gmail.com}



\begin{abstract}
We develop an algebraic framework over arbitrary quadratic fields $L = \mathbb{Q}(\sqrt{D})$ to generalize the Miller-Rabin primality test. Consequently, we present a deterministic primality test for integers of the form $N = K p^{\ell} - 1$ that requires only a single modular exponentiation and achieves a computational complexity of $\tilde{\mathcal{O}}(\log^2 N)$.  Furthermore, we also establish an analogue of Korselt's criterion within this setting. Finally, computational data generated using SageMath confirm its efficiency, successfully establishing the primality of numbers in the associated quadratic field within milliseconds.
\end{abstract}
\maketitle

\section{Introduction}

The problem of constructing efficient deterministic primality tests has long been a central topic in computational number theory. Classical tests based on Fermat's Little Theorem, together with its probabilistic refinement, the Miller--Rabin test, provide effective methods for detecting composite numbers and probable primes.  For integers of special forms, classical results of Lucas and Lehmer provide fast and deterministic methods for testing primality, such as the Lucas--Lehmer test for Mersenne numbers $2^n - 1$ and Proth's theorem for numbers of the form $k \cdot 2^n + 1$. These ideas have been extensively generalized to wider classes of integers (cf. \cite{MR1487359,MR2047101,Bos,D-H,MR3464611,G-O-S,G-O-A-S,MR2254741,S-W,t:hal-05192417,Wil3,MR638738,Wil,Wil2}).

Early generalizations were based on properties of Lucas sequences and focused on integers of the form $Ap^n-1$. For example, Williams ~\cite{Wil3} obtained necessary and sufficient conditions for the primality of integers of the form $2A3^n-1$. Further developments include tests for integers of the forms $A5^n-1$ and $A7^n-1$~\cite{Wil2}, as well as explicit criteria for $(p-1)p^n-1$~\cite{S-W}. In a different direction, Bosma ~\cite{Bos} established explicit primality criteria for integers of the form $h2^k\pm 1$ using a finite set of initial values of certain recurrence sequences.

In recent years, algebraic methods based on quadratic fields have been employed to study such problems. Grau et al.~\cite{G-O-A-S} developed primality tests for integers of the form $Kp^n+1$ requiring only a single modular exponentiation. They also introduced a Lucas-type test for integers of the form $4Kp^n-1$ using the group
\[
\mathcal{G}_N := \{ z \in (\mathbb{Z}/N\mathbb{Z})[i] : z\overline{z} \equiv 1 \pmod{N} \}
\]
over the Gaussian integers~\cite{G-O-S}. This approach was further extended by Sridhar et al.~\cite{t:hal-05192417}, who studied integers of the form $8kp^n-1$ via the structure of a subgroup defined over $\mathbb{Q}(\sqrt{-2})$.

In this article, we significantly extend these ideas to arbitrary quadratic fields $L = \mathbb{Q}(\sqrt{D})$. By determining the order of the associated unitary group over finite fields, we define a natural analogue of Euler's totient function in this setting. We also introduce a corresponding notion of Carmichael-type numbers and establish an analogue of Korselt's criterion.

As an application, we obtain a deterministic Lucasian-type primality test for integers of the form $N=c_D k p^\ell -1$. The test requires only a single modular exponentiation in the associated quotient ring and has complexity $\tilde{\mathcal{O}}(\log^2 N)$.

\section{Statement of Results}
In this article, we provide a criterion for testing the primality of odd integers of the form $N = c_D k p^{\ell} - 1$. Throughout this paper, let $L = \mathbb{Q}(\sqrt{D})$ be a quadratic field, where $D$ is a square-free integer. Unless otherwise specified, $p$ denotes an odd prime, and $k, \ell$ are positive integers. We assume $c_D$ is a fixed positive integer such that $\gcd(c_D k, p) = 1$ and the Jacobi symbol  satisfies $\left(\frac{D}{N}\right) = -1$. 

It is worth noting that in the case where $D < 0$ is a square-free integer, one may define
\[
c_D =
\begin{cases}
2|D| & \text{if } D \equiv 1 \pmod{4}, \\
4|D| & \text{if } D \equiv 2, 3 \pmod{4}.
\end{cases}
\]
With this specific choice of $c_D$, we have $N \equiv -1 \pmod{|D|}$. Consequently, by the properties of the Jacobi symbol, we obtain $\left(\frac{D}{N}\right) = -1$.

Let $\mathcal{O}_L$ denote the ring of integers of $L$. For an integer $n \ge 2$, the quotient ring $\mathcal{O}_L / n\mathcal{O}_L$ is isomorphic to the ring $\{a + b\sqrt{D} : a,b \in \mathbb{Z}/n\mathbb{Z}\}$ if $D \equiv 2, 3 \pmod{4}$. If $D \equiv 1 \pmod{4}$, it is isomorphic to $\{a + b\omega : a,b \in \mathbb{Z}/n\mathbb{Z}\}$, where $\omega = \frac{1+\sqrt{D}}{2}$. 

Based on this, we represent the quotient ring $\mathcal{O}_L / n\mathcal{O}_L$ by $\mathcal{I}_n(D)$ and define the corresponding set $\mathcal{G}_n(D)$ as follows. For $D \equiv 2,3 \pmod{4}$, 
$$
\mathcal{I}_n(D) := \{a + b\sqrt{D} : a,b \in \mathbb{Z}/n\mathbb{Z}\},
$$
and
$$
\mathcal{G}_n(D) := \{a + b\sqrt{D} \in \mathcal{I}_n(D) : a^2 - D b^2 \equiv 1 \pmod{n}\}.
$$
For $D \equiv 1 \pmod{4}$, 
$$
\mathcal{I}_n(D) := \{a + b\omega : a,b \in \mathbb{Z}/n\mathbb{Z}\},
$$
and
$$
\mathcal{G}_n(D) := \left\{a + b\omega \in \mathcal{I}_n(D) : a^2 + ab + \left(\frac{1 - D}{4}\right)b^2 \equiv 1 \pmod{n} \right\}.
$$

\noindent
\textbf{Remark.} For an odd integer $n>2$, $2$ is a unit in $\mathbb{Z}/n\mathbb{Z}$. Thus, the ring $\mathcal{O}_L / n\mathcal{O}_L$ is isomorphic to $\mathbb{Z}[\sqrt{D}]/n\mathbb{Z}[\sqrt{D}]$. Since our primality test applies exclusively to odd integers $n$, without loss of generality, we may define our algebraic structures entirely over the ring $\{a + b\sqrt{D} : a,b \in \mathbb{Z}/n\mathbb{Z}\}$.\\

With the above notations, we prove the following results.

\begin{thm}\label{T1.1}
Let $N = c_D k p^\ell - 1$ be an odd integer such that $c_Dk < p^{\ell}$. Suppose that there exists an element $w \in \mathcal{G}_N(D)$ such that 
$w^{\frac{N+1}{p}} \not\equiv 1 \pmod{N}$. Then $N$ is prime if and only if
$$
\Phi_p\left(w^{\frac{N+1}{p}}\right) \equiv 0 \pmod{N},
$$
where $\Phi_p(x)$ denotes the $p$-th cyclotomic polynomial.
\end{thm}
 
The following theorem provides a sufficient condition for primality.

\begin{thm}[Generalized Lucasian Certificate]\label{T1.2}
Let $N = c_D k p^\ell - 1$ be an odd integer and $w \in \mathcal{G}_N(D)$. Suppose there exists an integer $j$ with $1 \le j \le \ell$ such that:
\begin{enumerate}
    \item[(i)] $\Phi_p(w^{c_D k p^{j-1}}) \equiv 0 \pmod{N}$,
    \item[(ii)] $2j \ge \log_p(c_D k) + \ell$.
\end{enumerate}
Then $N$ is prime.
\end{thm}

One may ask whether condition $(i)$ of Theorem~\ref{T1.2} holds for some $j$. The following theorem guarantees that condition $(i)$ of Theorem~\ref{T1.2} holds whenever $N$ is prime.

 \begin{thm}[Generalized Miller--Rabin]\label{T1.3}
Let $N = c_D k p^\ell - 1$ be an odd prime and $w \in \mathcal{G}_N(D)$. Then one of the following holds:
\begin{enumerate}
    \item[(i)] $w^{c_D k} \equiv 1 \pmod{N}$;
    \item[(ii)] there exists an integer $j$ with $0 \leq j < \ell$ such that
    $\Phi_p\bigl(w^{c_D k p^j}\bigr) \equiv 0 \pmod{N}$.
\end{enumerate}
\end{thm}  
Motivated by the above theorem, we define the concept of a generalized strong probable prime as follows.

\noindent\textbf{Definition.}
Let $N = c_D k p^l - 1$ be an odd composite integer with $\gcd(N, D) = 1$, and let $w \in \mathcal{G}_N(D)$. We say that $N$ is a \emph{$p$-generalized strong probable prime to base $w$} if it satisfies either condition (i) or condition (ii) of Theorem \ref{T1.3}.\\

In a later section, we show that the group $\mathcal{G}_N(D)$ is cyclic of order $N \pm 1$ whenever $N$ is an odd prime. Hence, there exists an element $w \in \mathcal{G}_N(D)$ of order $N \pm 1$. Such an element does not satisfy condition~(i) of Theorem~\ref{T1.3}. Therefore, by Theorem~\ref{T1.3}, there exists an integer $j$ with $1 \le j \leq \ell$ such that $\Phi_p\bigl(w^{c_D k p^{j-1}}\bigr) \equiv 0 \pmod{N}$. This shows that condition~(i) of Theorem~\ref{T1.2} holds whenever $N$ is an odd prime.

The Miller--Rabin primality test applies to numbers of the form $2^s u + 1$, where $u$ is a positive odd integer and $s$ is a positive integer. We prove the following theorem, which helps test the primality of numbers of the form $2^s u - 1$, where $u$ is a positive odd integer and $s$ is a positive integer.
\begin{thm}\label{T1.4}
Let $N$ be an odd prime and let $D$ be a square-free integer such that $\left(\frac{D}{N}\right) = -1$. Write $N + 1 = 2^s u$, where $u$ is an odd integer and $s \ge 1$. For any $w \in \mathcal{G}_N(D) \setminus \{1, -1\}$, one of the following conditions must hold:
\begin{enumerate}
    \item[(a)] $w^u \equiv 1 \pmod N$.
    \item[(b)] $w^{ 2^ru} \equiv -1 \pmod N$ for some integer $0 \le r < s$.
\end{enumerate} 
\end{thm}
We remark that Theorem~\ref{T1.4} ensures that if an odd integer $N$ fails to satisfy both conditions of Theorem~\ref{T1.4} for a randomly chosen $w \in \mathcal{G}_N(D) \setminus \{1, -1\}$, then $N$ must be composite. We present the proof of Theorem~\ref{T1.4} in Section~5. In this section, we illustrate how to use Theorem~\ref{T1.4} to test the primality of an integer by means of an example.

The organization of this article is as follows. Section 3 contains key results that form the backbone of the paper and are used throughout. The proofs of Theorems~\ref{T1.1}, \ref{T1.2}, and \ref{T1.3} are presented in Section 4, The proof of Theorem~\ref{T1.4} is given in Section 5, along with illustrative examples. Section~6 describes the generalized primality testing algorithm and establishes its $\tilde{\mathcal{O}}(\log^2 N)$ computational complexity. In Section 7, we extend the notions of pseudoprimes and Carmichael numbers, provide a sufficient criterion for detecting composite numbers, and establish an analogue of Korselt's criterion. Section~8 presents the computational results and execution times of the algorithm over various quadratic fields. The final section contains concluding remarks.

\section{Preliminaries}
We prove the following theorem, which is crucial to this article.

\begin{thm} \label{T2.1}
Let $p$ be an odd prime such that $p \nmid D$. Then the order of the group $\mathcal{G}_p(D)$ is given by
$$
|\mathcal{G}_p(D)| = p - \left(\frac{D}{p}\right),
$$
where $\left(\frac{\cdot}{p}\right)$ denotes the Legendre symbol. Furthermore, the group $\mathcal{G}_p(D)$ is cyclic.
\end{thm}
\begin{proof}
Clearly, $D \equiv 1,2,3 \pmod{4}$. We first prove the theorem for the case $D \equiv 2,3 \pmod{4}$, and then for the case $D \equiv 1 \pmod{4}$.

\noindent
\textbf{Case 1:} Let $D \equiv 2,3 \pmod{4}$. Then
\[
\mathcal{G}_p(D) = \{a + b\sqrt{D} : a^2 - Db^2 \equiv 1 \pmod{p}\;\text{and}\; a,b \in \mathbb{Z}/p\mathbb{Z}\}.
\]
Thus, finding $|\mathcal{G}_p(D)|$ is equivalent to counting the number of solutions of
\begin{equation}\label{E1}
    x^2 - Dy^2 \equiv 1 \pmod{p}
\end{equation}
over the finite field $\mathbb{F}_p$.

\noindent
\textbf{Subcase 1:} Suppose $\left(\frac{D}{p}\right) = 1$. Then $D$ is a quadratic residue modulo $p$, so there exists $c \in \mathbb{F}_p$ such that $c^2 \equiv D \pmod{p}$. Using this in \eqref{E1}, we obtain
\begin{equation}\label{e1}
x^2 - c^2y^2 \equiv 1 \pmod{p}
\;\Longleftrightarrow\;
(x - cy)(x + cy) \equiv 1 \pmod{p}.
\end{equation}
Let $u = x - cy$ and $v = x + cy$. Since $p$ is odd and $c \not\equiv 0\pmod p$, this change of variables is bijective. Hence, \eqref{e1} is equivalent to
\[
uv \equiv 1 \pmod{p}.
\]
For each $u \in \mathbb{F}_p^\times$, there exists a unique $v = u^{-1}$. Therefore, the total number of solutions is $p-1$. Hence,
\[
|\mathcal{G}_p(D)| = p - 1 = p - \left(\frac{D}{p}\right).
\]

\noindent
\textbf{Subcase 2:} Suppose $\left(\frac{D}{p}\right) = -1$. Then $D$ is not a square in $\mathbb{F}_p$. This implies that $t^2 - D$ is irreducible over $\mathbb{F}_p$. Therefore, $\frac{\mathbb{F}_p[t]}{\langle t^2 - D \rangle} \cong \mathbb{F}p(\sqrt{D})$ is a field of order $p^2$, and this field is isomorphic to $\mathbb{F}_{p^2}$.

Consider the norm map $\mathcal{N} : \mathbb{F}_{p}(\sqrt{D})^\times \to \mathbb{F}_p^\times$ defined by
\[
\mathcal{N}(\alpha) = \alpha \cdot \alpha^p = \alpha^{p+1}.
\]
For $\alpha = a + b\sqrt{D}$, we have
\[
\mathcal{N}(a + b\sqrt{D}) = a^2 - Db^2.
\]
Thus, $\mathcal{G}_p(D)$ is exactly the kernel of the norm map. Since $\mathcal{N}$ is a surjective homomorphism, it follows that
\[
|\mathcal{G}_p(D)| = \frac{|\mathbb{F}_{p}(\sqrt{D})^\times|}{|\mathbb{F}_p^\times|}
= \frac{p^2 - 1}{p - 1} = p + 1.
\]
Hence,
\[
|\mathcal{G}_p(D)| = p + 1 = p - \left(\frac{D}{p}\right).
\]

\noindent
\textbf{Case 2:} Let $D \equiv 1 \pmod{4}$. Then
\[
\mathcal{G}_p(D) = \left\{a + b\omega : a^2 + ab + \left(\frac{1-D}{4}\right)b^2 \equiv 1 \pmod{p},\; a,b \in \mathbb{Z}/p\mathbb{Z}\right\},
\]
where $\omega = \frac{1 + \sqrt{D}}{2}$.

Thus, to determine $|\mathcal{G}_p(D)|$, it suffices to count the number of solutions to
\[
x^2 + xy + \left(\frac{1-D}{4}\right)y^2 \equiv 1 \pmod{p}
\]over the finite field $\mathbb{F}_p$.
Since $p$ is odd, multiplying both sides by $4$ gives
\begin{equation}\label{e2}
(2x + y)^2 - Dy^2 \equiv 4 \pmod{p}.
\end{equation}
Let $u = 2x + y$ and $v = y$. Then \eqref{e2} becomes
\[
u^2 - Dv^2 \equiv 4 \pmod{p}.
\]
Since $2 \not\equiv 0 \pmod{p}$, the transformation from $(x,y)$ to $(u,v)$ is bijective. Therefore, counting the solutions to the original equation is equivalent to counting the solutions to
\[
u^2 - Dv^2 \equiv 4 \pmod{p}.
\]
Setting $U = u/2$ and $V = v/2$, we obtain
\[
U^2 - DV^2 \equiv 1 \pmod{p}.
\]
Thus, the number of solutions is the same as in Case 1. Hence,
\[
|\mathcal{G}_p(D)| = p - \left(\frac{D}{p}\right).
\]
Finally, we prove that $\mathcal{G}_p(D)$ is cyclic. If $\left(\frac{D}{p}\right) = 1$, then $\mathcal{G}_p(D) \cong \mathbb{F}_p^\times$, which is cyclic. If $\left(\frac{D}{p}\right) = -1$, then $\mathcal{G}_p(D)$ is a subgroup of $\mathbb{F}_{p^2}^\times$, which is cyclic. Therefore, in both cases, $\mathcal{G}_p(D)$ is cyclic. This completes the proof.
\end{proof}
Using Chinese remainder theorem, it is easy to establish the following proposition.

\begin{prop}\label{P2.2}
Let $m, n \geq 2$ be coprime integers, and let $D$ be a square-free integer. Then there exists a canonical group isomorphism
\[
\mathcal{G}_{mn}(D) \simeq \mathcal{G}_m(D) \times \mathcal{G}_n(D).
\]
\end{prop}

    
\section{Proofs of Theorems \ref{T1.1}, \ref{T1.2} and \ref{T1.3}}
 
\begin{proof}[Proof of Theorem \ref{T1.1}] 
Let 
$X = w^{\frac{N+1}{p}}.$ Then 

\begin{equation} \label{e3}
    X^p = w^{N+1}.
\end{equation}
Assume that $N$ is a prime. Since $
\left(\frac{D}{N}\right) = -1$, by Theorem \ref{T2.1} we have $|\mathcal{G}_N(D)| = N + 1.$ Consequently, By Lagrange’s theorem,
\begin{equation}\label{e4}
w^{N+1} \equiv 1 \pmod{N}.
\end{equation}
From \eqref{e3} and \eqref{e4}, we get
\[
X^p \equiv 1 \pmod{N}.
\]
Thus,
\begin{equation}\label{eq3}
X^p - 1 = (X - 1)\Phi_p(X) \equiv 0 \pmod{N}.
\end{equation}
By hypothesis,
\[
X \not\equiv 1 \pmod{N}.
\]
Since $N$ is prime, the ring $\mathcal{I}_N(D)$ has no zero divisors. Therefore, from \eqref{eq3}, we must have
\[
\Phi_p(X) \equiv 0 \pmod{N}.
\]

\medskip

Conversely, assume that
\[
\Phi_p(X) \equiv 0 \pmod{N},
\]where $X = w^{\frac{N+1}{p}} = w^{c_Dkp^{\ell-1}}$.
Define
\[
\beta = w^{c_D k} \in \mathcal{G}_N(D).
\]
Then
\[
X = \beta^{p^{\ell-1}}.
\]
From $\Phi_p(X) \equiv 0 \pmod{N}$ and the identity $$x^p - 1 = (x-1) \Phi_p(x),$$ it follows that
\[
X^p \equiv 1 \pmod{N},
\]
and hence
\begin{equation}\label{eq4}
\beta^{p^\ell} \equiv 1 \pmod{N}.
\end{equation}

Suppose, for contradiction, that $N$ is composite. Let $q$ be a prime divisor of $N$ such that $q \le \sqrt{N}$. Reducing \eqref{eq4} modulo $q$, we obtain
\[
\beta^{p^\ell} \equiv 1 \pmod{q}.
\]
Thus, the $\mathrm{ord}_q(\beta)$ divides $p^\ell$, where $\mathrm{ord}_q(\beta)$ denotes the order of $\beta$ in   the group $\mathcal{G}_q(D)$. \\

If $\mathrm{ord}_q(\beta) = p^j$ with $j < \ell$, then
\[
\beta^{p^{\ell-1}} \equiv 1 \pmod{q},
\]
which implies
\begin{equation} \label{e7}
    X \equiv 1 \pmod{q}.
\end{equation}

We have
\[
\Phi_p(1) = p.
\]
Since $\Phi_p(X) \equiv 0 \pmod{q}$, using \eqref{e7} and the value of $\Phi_p(1)$ we obtain  $q = p$, which is impossible as
\[
N = c_D k p^\ell - 1.
\]
Thus,
\[
\mathrm{ord}_q(\beta) = p^\ell.
\]

Let $\beta_1 \equiv \beta \pmod{q}$. By Lagrange’s theorem,
\[
\mathrm{ord}_q(\beta_1) \mid |\mathcal{G}_q(D)| = q - \left(\frac{D}{q}\right) = q \pm 1.
\]
Hence,
\begin{equation}\label{eq5}
p^\ell \le q \pm 1.
\end{equation}
Since $p$ and $q$ are odd primes, $p^\ell \neq q+1$. Consequently,
\[
p^\ell < q.
\]

Using $q \le \sqrt{N}$, we get
\[
p^{2\ell} < q^2 \le N = c_D k p^\ell - 1,
\]
which implies
\[
p^{2\ell} < c_D k p^\ell.
\]
Dividing by $p^\ell$, we obtain
\[
p^\ell < c_D k,
\]
which contradicts the assumption $c_D k < p^\ell$. Therefore, $N$ must be prime. This completes the proof.
\end{proof}
\begin{proof}[Proof of Theorem \ref{T1.2}]
Let $X = w^{c_D k}$. By condition (i), we have
\begin{equation}\label{e9}
    \Phi_p(X^{p^{j-1}}) \equiv 0 \pmod{N},
\end{equation}
for some integer $j$ with $1 \le j \le \ell$. Fix such an integer $j$ satisfying $1 \le j \le \ell$ for which \eqref{e9} holds.\\

Suppose, for the sake of contradiction, that $N$ is composite. Let $q$ be a prime divisor of $N$ such that $q \le \sqrt{N}$. Reducing \eqref{e9} modulo $q$, we obtain
$\Phi_p(X^{p^{j-1}}) \equiv 0 \pmod{q}$. From this and the identity $$x^p - 1 = (x-1) \Phi_p(x),$$ we have $X^{p^j} \equiv 1 \pmod{q}$. This concludes that 
the multiplicative order of $X$ in $\mathcal{G}_q(D)$ divides $p^j$. If $\mathrm{ord}_q(X) = p^s$ for some $s < j$, then
$X^{p^s} \equiv 1 \pmod{q}$,
which implies
$$\Phi_p(1) = p \equiv \Phi_p(X^{p^s}) \equiv \Phi_p(X^{p^{j-1}}) \equiv 0 \pmod{q}.$$
This forces $q = p$, which is impossible since $q \mid N = c_D k p^\ell - 1$. Therefore,
$\mathrm{ord}_q(X) = p^j$.

By Lagrange’s theorem, we must have
$p^j \mid |\mathcal{G}_q(D)| = q \pm 1$.
Since $p$ and $q$ are odd primes, it follows that
$p^j < q \le \sqrt{N}$.
Squaring both sides, we obtain
$p^{2j} < N < c_D k p^\ell$.
Taking logarithms to the base $p$, we obtain
\begin{equation}\label{e10}
    2j < \log_p(c_D k) + \ell.
\end{equation}
This inequality holds for each fixed $j$ with $1 \le j \le \ell$ for which \eqref{e9} holds. 
Thus, \eqref{e10} contradicts condition $(ii)$. Therefore, $N$ must be prime.
\end{proof}
\begin{proof}[Proof of Theorem \ref{T1.3}]
Since $N$ is an odd prime and $\left(\frac{D}{N}\right) = -1$, Theorem \ref{T2.1} implies that $|\mathcal{G}_N(D)| = N + 1$. Hence, 
\begin{equation}\label{e11}
  w^{N+1} = w^{c_D k p^\ell} \equiv 1 \pmod{N}.  
\end{equation}

Consider the sequence
\[
w^{c_D k}, \; w^{c_D k p}, \; \dots, \; w^{c_D k p^\ell}.
\]
If $w^{c_D k} \equiv 1 \pmod{N}$, then $(i)$ holds. Otherwise, by \eqref{e11} the last term of the sequence is $1$ modulo $N$. Therefore, there exists a smallest integer $r$ with $1 \leq r \leq \ell$ such that
\[
w^{c_D k p^r} \equiv 1 \pmod{N}.
\]

Using the factorization $x^p - 1 = (x - 1)\Phi_p(x)$ with $x = w^{c_D k p^{r-1}}$, we obtain
\[
w^{c_D k p^r} - 1 = \bigl(w^{c_D k p^{r-1}} - 1\bigr)\Phi_p\bigl(w^{c_D k p^{r-1}}\bigr) \equiv 0 \pmod{N}.
\]

By the minimality of $r$, we have
\[
w^{c_D k p^{r-1}} \not\equiv 1 \pmod{N}.
\]
Since $N$ is prime, it follows that
\[
\Phi_p\bigl(w^{c_D k p^{r-1}}\bigr) \equiv 0 \pmod{N}.
\]Setting $j = r - 1$, we obtain $0 \leq j < \ell$, which gives condition (ii). This finishes the proof.
\end{proof}

\section{Proof of Theorem \ref{T1.4} and some Examples}
\begin{proof}[Proof of Theorem \ref{T1.4}]
As $N$ is an odd prime and $\left(\frac{D}{N}\right) = -1$, Theorem~\ref{T2.1} implies that the group $\mathcal{G}_N(D)$ is cyclic of order $N+1$. Moreover, $N+1$ is even, so there exist an odd integer $u$ and a positive integer $s$ such that $N+1 = 2^s u$. Thus, we have
\begin{equation}\label{e12}
    w^{ 2^s \cdot u} = w^{N+1} \equiv 1 \pmod{N}.
\end{equation}

Consider the sequence $w^u,\, w^{2u},\, w^{2^2 u},\, \dots,\, w^{2^s u}$. 
If $w^u \equiv 1 \pmod{N}$, then condition (a) holds. Otherwise, by \eqref{e12} the last term of the sequence is $1$ modulo $N$. Therefore, there exists a smallest integer $r+1 \le s$ such that
$w^{ 2^{r+1}u} \equiv 1 \pmod{N}$. By the minimality of $r+1$, we have
$w^{2^ru} \not\equiv 1 \pmod{N}$.
Since $N$ is prime, $\mathbb{Z}_N[\sqrt{D}]$ is a field. Hence, the equation $x^2 \equiv 1 \pmod{N}$ has exactly two solutions, namely $\pm 1$. Therefore,
$(w^{ 2^ru})^2 \equiv 1 \pmod{N}$ and $w^{2^ru} \not\equiv 1 \pmod{N}$ together imply that
$w^{2^ru} \equiv -1 \pmod{N}$. Thus, condition (b) holds, and the proof is complete.
\end{proof}
We now illustrate how Theorem~\ref{T1.4} can be used to test the primality of integers through the following example.\\

\noindent
\textbf{Example.}
Let $N = 35$ and $D = -3$. Then, $\left(\frac{-3}{35}\right) = -1$. Suppose, for the sake of contradiction, that $N$ is an odd prime; then Theorem~\ref{T2.1} implies $|\mathcal{G}_{35}(-3)| = 36 = 9 \cdot 2^2$. Thus, $u = 9$ and $s = 2$.

Choose $w = 3 + 3\sqrt{-3} \in \mathcal{G}_{35}(-3)$. A direct computation shows that $w^9 \equiv 29 \pmod{35}$, which is not congruent to $\pm 1$, and $w^{18} \equiv 29^2 \equiv 1 \pmod{35}$. Thus, $w^u \not\equiv 1 \pmod{35}$ and $w^{u \cdot 2^r} \not\equiv -1 \pmod{35}$ for all $r$ satisfying $0 \leq r < 2$. Therefore, by Theorem~\ref{T1.4}, we reach a contradiction, and hence $N$ must be composite.
\section{Algorithm and Complexity} 
Using the factorization $\Phi_p(x)(x - 1) = x^p - 1$, the cyclotomic condition $\Phi_p\bigl(w^{c_D k p^{j-1}}\bigr) \equiv 0 \pmod{N}$ can be expressed in a computationally efficient form as
$w^{c_D k p^j} \equiv 1 \pmod{N}$ and $\gcd\bigl(w^{c_D k p^{j-1}} - 1, N\bigr) = 1$.

To initialize the test, we efficiently generate a base element $w \in \mathcal{G}_N(D)$. Instead of utilizing a brute-force search for elements satisfying the unitary norm equation, we can select a non-zero element $z \in \mathcal{I}_N(D) \setminus \{0\}$ and set $w \equiv z/\overline{z} \pmod N$. If $N$ is prime, $\mathcal{I}_N(D)$ forms a finite field of order $N^2$, meaning every non-zero element is invertible. Thus, a uniformly chosen $z \in \mathcal{I}_N(D) \setminus \{0\}$ provides an invertible base element with probability $1$ whenever $N$ is prime, requiring $\mathcal{O}(1)$ expected operations. This establishes the following generalized primality testing algorithm.

\vspace{0.3cm}
\noindent \textbf{Algorithm 1: Generalized Primality Test over $\mathcal{G}_{N}(D)$}
\begin{enumerate}
    \item[(1)] \textbf{Input:} An odd prime $p$, and integers $k, {\ell} \ge 1$ such that $\gcd(c_D k, p) = 1$. Set $N = c_D k p^{\ell} - 1$.
    \item[(2)] \textbf{Generate base $w$:} Choose a non-zero element $z \in \mathcal{I}_N(D)$.
    \item[(3)] Compute the norm $\mathcal{N}(z) = z\overline{z} \pmod N$. If $\gcd(\mathcal{N}(z), N) \neq 1$, return ``$N$ is composite''.
    \item[(4)] Compute $w \equiv z / \overline{z} \pmod N$ and set $S_0 := w^{c_D k} \pmod N$.
    \item[(5)] If $S_0 \equiv 1 \pmod N$, return ``$N$ is a $p$-generalized strong probable prime''.
    \item[(6)] For $i=1$ to $\ell$ do:
    \item[(7)] \quad Compute $S_i := S_{i-1}^p \pmod N$.
    \item[(8)] \quad If $S_i \equiv 1 \pmod N$, then:
    \item[(9)] \quad \quad If $\gcd(S_{i-1}-1, N) = 1$, then:
    \item[(10)] \quad \quad \quad Set $j := i$ and go to Step 17.
    \item[(11)] \quad \quad Else:
    \item[(12)] \quad \quad \quad return ``$N$ is composite''.
    \item[(13)] \quad \quad end if
    \item[(14)] \quad end if
    \item[(15)] end for
    \item[(16)] return ``$N$ is composite''.
    \item[(17)] If $2j < \log_p(c_D k) + \ell$, then:
    \item[(18)] \quad return ``$N$ is a $p$-generalized strong probable prime''.
    \item[(19)] Else:
    \item[(20)] \quad return ``$N$ is prime''.
\end{enumerate}
\vspace{0.3cm}

\begin{thm}
For fixed $p$ and $k$, the computational complexity of Algorithm~1 is $\tilde{\mathcal{O}}(\log^2 N)$.
\end{thm}

\begin{proof}
The most expensive operations occur inside the loop. The algorithm performs at most $\ell$ modular exponentiations in the quotient ring $\mathcal{I}_N(D)$. Using fast multiplication (e.g., the Sch{\"o}nhage--Strassen algorithm), a single exponentiation requires $\tilde{\mathcal{O}}(\log N)$ bit operations. Since $\ell = \log_p\!\left(\frac{N+1}{c_D k}\right)$, the number of iterations grows linearly with $\log N$. Therefore, the total complexity of exponentiation is $\tilde{\mathcal{O}}(\log^2 N)$.

The greatest common divisor computation requires $\mathcal{O}(\log N)$ operations and is negligible compared to the exponentiation complexity. This completes the proof.
\end{proof}

\section{Pseudoprime and Carmichael Numbers}
We begin this section with the following definition.

\begin{defn}
For an integer $n \ge 2$ with $\gcd(n, D) = 1$, we define the arithmetic function:
\[ 
\mathcal{F}_D(n) := n - \left(\frac{D}{n}\right), 
\]
where $\left(\frac{D}{n}\right)$ denotes the Kronecker symbol.
\end{defn}
Since $|\mathcal{G}_n(D)| = \mathcal{F}_D(n)$ when $n$ is an odd prime, the following result is immediate.

\begin{prop}\label{P5.2}
Let $p$ be an odd prime and let $\alpha \in \mathcal{G}_p(D)$. Then,
\begin{equation}\label{e13}
\alpha^{\mathcal{F}_D(p)} \equiv 1 \pmod{p}.
\end{equation}
\end{prop} 
Proposition \ref{P5.2} ensures that an odd integer $N$ is composite if 
\begin{equation}\label{e14}
    \alpha^{\mathcal{F}_D(N)} \not\equiv 1\pmod{N}\; \text{ for some}\; \alpha \in \mathcal{G}_N(D).
\end{equation} 

In \cite{t:hal-05192417}, the authors showed that for $N = 2737$, the identity \eqref{e13} holds for all $\alpha \in \mathcal{G}_N(-2)$, thus failing to detect that $N$ is composite. However, we can successfully identify $N = 2737$ as composite by selecting $D=3$.

In this case, the Jacobi symbol is $\left(\frac{3}{2737}\right) = 1$, which yields $\mathcal{F}_3(2737) = 2736$. By choosing the specific element $\alpha = 2 + \sqrt{3} \in \mathcal{G}_{2737}(3)$, a direct computation using SageMath confirms that
$$
    (2+\sqrt{3})^{2736} \not\equiv 1 \pmod{2737}.
$$
 Thus, Proposition \ref{P5.2} ensures that $N=2737$ is composite. \\

The significance of Proposition \ref{P5.2} is therefore as follows: for an odd integer $N$, if the identity \eqref{e13} fails for some element of any group $\mathcal{G}_N(D)$, then $N$ is composite. Hence, Proposition \ref{P5.2} generalizes Theorem 3.2 of \cite{t:hal-05192417} and offers potential applications in primality testing.

Now, we introduce some definitions to prove a result analogous to Korselt's criterion, which states that a composite integer $N$ is a Carmichael number if and only if:
\begin{enumerate}
    \item $N$ is square-free, and
    \item for every prime $p \mid N$, we have $p - 1 \mid N - 1$.
\end{enumerate}

\begin{defn}
Let $N$ be an odd composite integer with $\gcd(N, D)=1$. Then $N$ is called a \textbf{$\mathcal{G}(D)$-pseudoprime to base $\alpha$}, where $\alpha \in \mathcal{G}_N(D)\setminus\{1,-1\}$, if $\alpha^{\mathcal{F}_D(N)} \equiv 1 \pmod{N}$.
\end{defn}

\begin{defn}
Let $N$ be an odd composite integer with $\gcd(N, D)=1$. Then $N$ is called a \textbf{$\mathcal{G}(D)$-Carmichael number} if it is a $\mathcal{G}(D)$-pseudoprime to every base $\alpha \in \mathcal{G}_N(D)$.
\end{defn}
We prove the following result.
\begin{thm} \label{T4.3}
Let $N>1$ be an odd square-free composite integer such that $\gcd(N, D)=1$. Then $N$ is a $\mathcal{G}(D)$-Carmichael number if and only if $\mathcal{F}_D(p) \mid \mathcal{F}_D(N)$ for every prime $p \mid N$.
\end{thm}

\begin{proof}
Let $N = p_1 p_2 \cdots p_r$, where $p_1, p_2, \dots, p_r$ are distinct primes. Suppose that $N$ is a $\mathcal{G}(D)$-Carmichael number. Then $\alpha^{\mathcal{F}_D(N)} \equiv 1 \pmod{N}$ for all $\alpha \in \mathcal{G}_N(D)$. By Theorem \ref{T2.1}, the group $\mathcal{G}_{p_i}(D)$ is cyclic for each $i$. Let $\alpha_i$ be a generator of $\mathcal{G}_{p_i}(D)$. By Proposition \ref{P2.2}, there exists $\alpha \in \mathcal{G}_N(D)$ such that $\alpha \equiv \alpha_i \pmod{p_i}$ for all $i$. Since $N$ is a $\mathcal{G}(D)$-Carmichael number, we have $\alpha^{\mathcal{F}_D(N)} \equiv 1 \pmod{N}$, which implies $\alpha_i^{\mathcal{F}_D(N)} \equiv 1 \pmod{p_i}$. As $\alpha_i$ has order $\mathcal{F}_D(p_i)$, it follows that $\mathcal{F}_D(p_i) \mid \mathcal{F}_D(N)$ for all $i$.

 Conversely, assume that $\mathcal{F}_D(p_i) \mid \mathcal{F}_D(N)$ for all $i$. Let $\alpha \in \mathcal{G}_N(D)$, and  let $\alpha_i' \equiv \alpha \pmod{p_i}$. Since the order of $\mathcal{G}_{p_i}(D)$ is $\mathcal{F}_D(p_i)$ by Theorem \ref{T2.1}, we have ${\alpha_i'}^{\mathcal{F}_D(p_i)} \equiv 1 \pmod{p_i}$. Since $\mathcal{F}_D(p_i) \mid \mathcal{F}_D(N)$, there exists an integer $m'$ such that $\mathcal{F}_D(N) = m'\cdot \mathcal{F}_D(p_i)$. Hence, $\alpha^{\mathcal{F}_D(N)} \equiv ({\alpha_i}'^{\mathcal{F}_D(p_i)})^{m'} \equiv 1 \pmod{p_i}$. This holds for each $p_i$ and $N$ is square-free. Therefore,  $\alpha^{\mathcal{F}_D(N)} \equiv 1 \pmod{N}$. Hence, $N$ is a $\mathcal{G}(D)$-Carmichael number. This finishes the proof.
\end{proof}
 Theorem~\ref{T4.3} is a generalization of Theorem~3.11 of~\cite{t:hal-05192417}. 

\section{Computational Results}

To validate the applicability of Algorithm~1 and its $\tilde{\mathcal{O}}(\log^2 N)$ theoretical complexity, we implemented the algorithm in SageMath. We performed experiments over three different quadratic fields to show that it is effective in both real and imaginary settings.

Table \ref{tab:benchmark} summarizes the execution times for primes discovered during this search. For the imaginary cases, we tested $D=-2$ ($N = 8 \cdot 3^\ell - 1$) and $D=-3$ ($N = 6 \cdot 5^\ell - 1$). Crucially, to demonstrate the algorithm's applicability over real quadratic fields ($D > 0$), we tested the field $L = \mathbb{Q}(\sqrt{5})$ by setting $D=5$ and $c_D=2$, successfully proving primality for integers of the form $N = 2 \cdot 3^\ell - 1$. The implementation generated the necessary base elements in expected $\mathcal{O}(1)$ time and proved primality deterministically in milliseconds across all structures.

{
\renewcommand{\arraystretch}{1.2}
\begin{longtable}{cccccc}
    \caption{Execution times for deterministic primality proving using Algorithm 1 across different fields.}
    \label{tab:benchmark} \\
    
    \hline
    \textbf{Field $D$} & \textbf{Value of $l$} & \textbf{Bit-size of $N$} & \textbf{Prime $p$} & \textbf{Result} & \textbf{Time (seconds)} \\
    \hline
    \endfirsthead
    
    \multicolumn{6}{c}{{\tablename\ \thetable{} -- continued from previous page}} \\
    \hline
    \textbf{Field $D$} & \textbf{Value of $l$} & \textbf{Bit-size of $N$} & \textbf{Prime $p$} & \textbf{Result} & \textbf{Time (seconds)} \\
    \hline
    \endhead
    
    \hline \multicolumn{6}{r}{{Continued on next page...}} \\
    \endfoot
    
    \hline
    \endlastfoot
    
    $-2$ & 10  & 19  & 3 & $N$ is Prime & 0.00076 \\
    $-2$ & 17  & 30  & 3 & $N$ is Prime & 0.00040 \\
    $-2$ & 50  & 83  & 3 & $N$ is Prime & 0.00069 \\
    $-2$ & 170 & 273 & 3 & $N$ is Prime & 0.00209 \\
    $-2$ & 184 & 295 & 3 & $N$ is Prime & 0.00364 \\
    $-2$ & 194 & 311 & 3 & $N$ is Prime & 0.00295 \\
    \hline
    $-3$ & 2   & 8   & 5 & $N$ is Prime & 0.00091 \\
    $-3$ & 5   & 15  & 5 & $N$ is Prime & 0.00029 \\
    $-3$ & 11  & 29  & 5 & $N$ is Prime & 0.00042 \\
    $-3$ & 28  & 68  & 5 & $N$ is Prime & 0.00089 \\
    $-3$ & 65  & 154 & 5 & $N$ is Prime & 0.00130 \\
    $-3$ & 72  & 170 & 5 & $N$ is Prime & 0.00137 \\
    \hline
    $5$  & 2   & 5   & 3 & $N$ is Prime & 0.00022 \\
    $5$  & 3   & 6   & 3 & $N$ is Prime & 0.00022 \\
    $5$  & 7   & 13  & 3 & $N$ is Prime & 0.00044 \\
    $5$  & 23  & 38  & 3 & $N$ is Prime & 0.00046 \\
    $5$  & 27  & 44  & 3 & $N$ is Prime & 0.00061 \\
    $5$  & 35  & 57  & 3 & $N$ is Prime & 0.00119 \\
    $5$  & 62  & 100 & 3 & $N$ is Prime & 0.00147 \\
    $5$  & 131 & 209 & 3 & $N$ is Prime & 0.00199 \\
\end{longtable}
}

\section{Concluding Remarks}
Let $N = 2k p^{\ell}-1$ be an integer, where $p$ is an odd prime and $\ell, k$ are positive integers such that $p \nmid k$. Then there always exists a prime $Q$ such that $\left(\frac{Q}{N}\right) = -1$. The existence of such a prime $Q$ is guaranteed by~\cite{MR4470173} together with the law of quadratic reciprocity for odd positive integers. 

Henceforth, let $Q_N$ denote a prime satisfying $\left(\frac{Q_N}{N}\right) = -1$. Then the order of the group $\mathcal{G}_N(Q_N)$ is $N+1$ whenever $N$ is an odd prime. Therefore, using the method of this article, one can prove the following analogues of Theorems~\ref{T1.1}, \ref{T1.2}, and \ref{T1.3}.

\begin{thm}
Let $N = 2k p^{\ell} - 1$ be an integer, where $p$ is an odd prime and $\ell, k$ are positive integers such that $p\nmid k$ and $2 k < p^\ell$. Suppose there exists an element $w \in \mathcal{G}_N(Q_N)$ such that 
$w^{\frac{N+1}{p}} \not\equiv 1 \pmod{N}$. Then $N$ is prime if and only if
$$
\Phi_p\left(w^{\frac{N+1}{p}}\right) \equiv 0 \pmod{N},
$$
where $\Phi_p(x)$ denotes the $p$-th cyclotomic polynomial.
\end{thm}

\begin{thm}
Let $N = 2k p^{\ell} -1$ and  $w \in \mathcal{G}_N(Q_N)$. Suppose there exists an integer $j$ with $1 \le j \le \ell$ such that:
\begin{enumerate}
    \item[(i)] $\Phi_p(w^{2k p^{j-1}}) \equiv 0 \pmod{N}$,
    \item[(ii)] $2j \ge \log_p( 2k) + \ell$.
\end{enumerate}
Then $N$ is prime.
\end{thm}
\begin{thm}
Let $N = 2k p^{\ell} -1$ be an odd prime. For each $w \in \mathcal{G}_N(Q_N)$, one of the following holds:
\begin{enumerate}
    \item[(i)] $w^{2k} \equiv 1 \pmod{N}$;
    \item[(ii)] there exists an integer $j$ with $0 \leq j < \ell$ such that
    $\Phi_p\bigl(w^{2k p^j}\bigr) \equiv 0 \pmod{N}$.
\end{enumerate}
 \end{thm}   

\bibliographystyle{amsplain}
\bibliography{references}
\end{document}